\documentclass[11pt,a4paper]{article}

\usepackage{comment}
\usepackage{graphicx}
\usepackage{amsmath}
\usepackage{amsfonts}
\usepackage{amssymb}
\usepackage{enumerate}

\usepackage{float}
\newfloat{Test}{htbp}{lot}
\floatname{Test}{Test}

\usepackage{lscape}
\usepackage{longtable}
\usepackage{rotating}
\usepackage{color}

\usepackage{tabularx}
\usepackage{array}
\newcolumntype{Y}{>{\centering\arraybackslash}X}

\usepackage{url}
\usepackage{rotating}
\usepackage{algpseudocode}
\usepackage[ruled]{algorithm}
\usepackage{algcompatible}
\usepackage{pgf,tikz}
\usepackage{subfig}
\usepackage[titletoc]{appendix}
\usepackage{wrapfig}
\usepackage{adjustbox}
\usepackage{mathrsfs}
\usepackage{pgfplots}
\usetikzlibrary{arrows}
\usepackage{color}

\usepackage{bbm}

\usepackage{booktabs}
\usepackage{multirow}
\usepackage{mathtools}
\usepackage{makecell}

\usepackage[sort,numbers]{natbib}
\usepackage{eqparbox}%

\numberwithin{table}{section}
\numberwithin{figure}{section}
\numberwithin{equation}{section}%

\usepackage[symbol]{footmisc}%

\newtheorem{theorem}{Theorem}[section]
\newtheorem{proposition}{Proposition}[section]

\newtheorem{definition}{Definition}

\newtheorem{lemma}[theorem]{Lemma}

\newtheorem{problem}{Problem}[section]
\newtheorem{sht}{Test}[section]

\newtheorem{assumption}{Assumption}[section]

\newenvironment{proof}[1][Proof]{\noindent \textbf{#1.} }{\hfill$\Box$\par\medskip}

\newcommand{\beqn}[1]{\begin{equation}\label{#1}}
\newcommand{\eeqn}{\end{equation}}

\definecolor{darkgreen}{rgb}{0,0.6,0}
\definecolor{aau2}{rgb}{0.0, 0.5, 0.69}
\definecolor{aau3}{rgb}{0.0, 0.53, 0.74}
\definecolor{aau4}{rgb}{0.0, 0.48, 0.65}
\definecolor{aau5}{rgb}{0.0, 0.45, 0.73}
\definecolor{rsap}{RGB}{130, 36, 51}
\definecolor{gsap}{RGB}{112, 164, 137}
\definecolor{tud}{rgb}{0.43,0.73,0.11}
\definecolor{verde}{rgb}{0.33,0.53,0.11}

\definecolor{ttffqq}{rgb}{0.0, 0.48, 0.65} %{rgb}{0.43,0.73,0.11}
\definecolor{ffqqqq}{rgb}{0.0, 0.5, 0.69} %{rgb}{1,0,0}
\usetikzlibrary{arrows}

\tikzstyle{decision} = [diamond, draw, fill=blue!20,
text width=4.5em, text badly centered, node distance=3cm, inner sep=0pt]
\tikzstyle{block} = [rectangle, draw, fill=blue!20,
text centered, rounded corners, minimum height=4em]
\tikzstyle{line} = [draw, -latex']
\tikzstyle{cloud} = [draw, ellipse,fill=red!20, node distance=3cm,
minimum height=2em]
\tikzstyle{cloud2} = [draw, ellipse,fill=green!20, node distance=3cm,
minimum height=2em]
\begin{document}
	
\title{Sequential test sampling for stochastic derivative-free optimization}

	\author{
		A. Ding\thanks{Department of Industrial and Systems Engineering, Lehigh University, Bethlehem, PA 18015-1582, USA ({\tt and523@lehigh.edu}).}
		\and
		F. Rinaldi \thanks{Dipartimento di Matematica ``Tullio Levi-Civita'', Universit\`a
			di Padova, Italy
			(\tt{rinaldi@math.unipd.it}).}
	\and
	L. N. Vicente\thanks{Department of Industrial and Systems Engineering, Lehigh University, Bethlehem, PA 18015-1582, USA ({\tt lnv@lehigh.edu}).}
	}
	
	\maketitle

\begin{abstract}
In many derivative-free optimization algorithms, a sufficient decrease condition decides whether to accept a trial step in each iteration. This condition typically requires that the potential objective function value decrease of the trial step, 
i.e., the true reduction in the objective function value that would be achieved by moving from the current point to the trial point, be larger than a multiple of the squared stepsize.
When the objective function is stochastic, evaluating such a condition accurately can require a large estimation cost.

In this paper, we frame the evaluation of the sufficient decrease condition in a stochastic setting as a hypothesis test problem and solve it through a sequential hypothesis test. The two hypotheses considered in the problem correspond to accepting or rejecting the trial step. This test sequentially collects noisy sample observations of the potential decrease until their sum crosses either a lower or an upper boundary depending on the noise variance and the stepsize. 
When the noise of observations is Gaussian, we derive a novel sample size result, showing that the effort to evaluate the condition explicitly depends on the potential decrease, and that the sequential test terminates early whenever the sufficient decrease condition is away from satisfaction. Furthermore, when the potential decrease is~$\Theta(\delta^r)$ for some~$r\in(0,2]$, the expected sample size decreases from~$\Theta(\delta^{-4})$ to~$O(\delta^{-2-r})$.

We apply this sequential test sampling framework to probabilistic-descent direct search. To analyze its convergence rate, we extend a renewal-reward supermartingale-based convergence rate analysis framework to an arbitrary probability threshold. By doing so, we are able to show that probabilistic-descent direct search has an iteration complexity of $O(n/\epsilon^2)$ for gradient norm.
Our numerical experiments indicate the superiority of sequential hypothesis testing over fixed sampling when dealing with the evaluation of stochastic sufficient decrease conditions.
 \end{abstract}

\section{Introduction}\label{s1}
In this paper, we consider an unconstrained optimization problem of the form
\begin{align*}
    \min_{x\in\mathbb{R}^n}f(x),
\end{align*}
where the function values of the objective function $f:\mathbb{R}^n\rightarrow\mathbb{R}$ are not available directly. Instead, we have access to $F(x,\xi)$ as a noisy observation of $f(x)$, where $\xi$ is a random variable. We assume that the observed noise is unbiased
\begin{align}
    E_\xi[F(x,\xi)]=f(x).\label{eq1}
\end{align}
We also assume that the derivatives of $f$ are not available or that the cost of computing them is unaffordable. Furthermore, the objective function $f$ is assumed to be continuously differentiable and bounded from below by $f^*$. Its gradient $\nabla f$ is assumed to be $L_f$-Lipschitz continuous.
The non-availability of derivatives is a common scenario in many simulation-based optimization applications and is the main subject of derivative-free optimization (DFO). 
In DFO, the only information available about $f$ is through zeroth-order oracles (in our case, stochastic zeroth-order oracles), and the cost of querying the zeroth-order oracle is expensive. The goal of DFO is to achieve good solutions with few zeroth-order oracle queries.

Trust-region, direct-search, and line-search methods are three popular classes of algorithms in DFO. Trust-region methods build local surrogate models of $f$ and compute trial steps and decide its acceptance using those models, while direct search explores the space directly via a set of search directions or points, without explicit models. In the middle of the spectrum between with and without models, line-search methods approximate the steepest descent direction using finite differences and searches along it. 
For a more comprehensive understanding of these classes of DFO methods, interested readers are encouraged to consult sources such as, e.g., \cite{conn2009introduction,larson2019derivative}.
Most instances of these algorithms rely on a decrease condition to ensure that each step taken by the algorithm makes meaningful progress toward reducing the objective function.
In particular, probabilistic-descent direct search is a direct-search algorithm proposed in~\cite{gratton2015direct}, that instead of
using a positive spanning set (see, for example,~\cite{conn2009introduction}) to ensure a descent direction (which may require at least~$n+1$ function evaluations), 
incorporates randomness into the algorithm to obtain a descent direction probabilistically using only one point. It is shown in~\cite{gratton2015direct} that, when the objective function is deterministic, this algorithm has a zeroth-order oracle complexity of $O(n/\epsilon^2)$ for gradient norm. 
In this paper, we consider probabilistic-descent direct search for the purpose of applying sequential hypothesis testing to the evaluation of sufficient decrease conditions in stochastic DFO.

\subsection{A brief literature review of stochastic derivative-free optimization}\label{s11}

We start by giving a brief overview of the main results in the literature on stochastic derivative-free optimization. To our knowledge, they all require, under the standard assumption of noise   exhibiting finite variance, a sample size of $\Theta(\delta^{-4})$ function estimates per iteration, where $\delta$ is a stepsize or a trust-region radius (see, e.g., \cite{rinaldi2024stochastic} and references therein for further details on this matter).

A trust-region method is designed in~\cite{larson2016stochastic} to address optimization problems involving noisy objective functions. The authors establish convergence guarantees under conditions where the objective function $f$ exhibits adequate smoothness properties (specifically, possessing a Lipschitz continuous gradient) and when the noise is independently drawn from a distribution with zero mean and finite variance. ASTRO-DF, proposed in~\cite{shashaani2018astro} and refined in~\cite{ha2025iteration}, is an adaptive sampling trust-region method designed for objective functions that maintain Lipschitz continuous gradients and can be accessed through a Monte Carlo oracle. The framework in~\cite{shashaani2018astro} assumes noise that is independently distributed with zero mean, finite variance, and a bounded 4$\nu$th moment (where $\nu\geq 2$), and develops an almost sure convergence result.

Furthermore,~\cite{chen2018stochastic} provides another significant contribution by examining a trust-region algorithm for unconstrained stochastic optimization. This work focuses on random models derived from a smooth objective function using stochastic observations of either the function itself or its gradient. The convergence analysis and rates for such methodologies are further detailed in~\cite{blanchet2019convergence} using martingale theory, which is applicable to a wide class of stochastic algorithms including direct search. The theoretical frameworks developed in~\cite{blanchet2019convergence,cartis2018global,chen2018stochastic} represent extensions of the probabilistic trust-region DFO approach originally outlined in~\cite{bandeira2014convergence} for deterministic functions. Each of these trust-region algorithms requires functions to possess some level of smoothness (such as Lipschitz continuous gradients) and relies on the ability to construct probabilistically accurate gradient approximations.

A comprehensive review of stochastic direct-search variants is provided in the survey~\cite[Chapter 4]{dzahini2025direct} for both smooth and non-smooth objective functions. StoMADS, proposed in~\cite{audet2021stochastic}, is a stochastic variant of the mesh adaptive direct search (MADS) algorithm. StoMADS generates an asymptotically dense set of search directions and is proved in~\cite{audet2021stochastic} using martingale theory to converge to a Clarke stationary point of a locally Lipschitz continuous function with probability one. In another line of research, the work in~\cite{dzahini2022expected} considers stochastic direct-search methods of directional type and presents its convergence rate analysis by utilizing the supermartingale-based framework in~\cite{blanchet2019convergence}. In~\cite{rinaldi2024stochastic}, a new probabilistic tail-bound condition for function estimation is introduced under which stochastic direct-search and trust-region methods are shown to converge globally. Reduction in sample complexity is obtained under stronger assumptions than the standard finite noise variance. More specifically, 
under a bounded  $q/(1-q)$~moment assumption, using $c\delta^q$ with $q\in(1,2]$ as a threshold for decrease in the acceptance test, 
the authors give a $O(\delta^{-2q})$ sample complexity. Furthermore, under the assumption of using a common number generator framework and correlated errors (satisfied when, e.g., the noise is modeled as a Gaussian process), the authors give a $O(\delta^{2-2q})$ instead.  As noted in \cite[Remark 5.1]{rinaldi2024stochastic}, the improvement in the number of samples per iteration does not however necessarily lead to a reduction in the overall computational cost of the considered algorithmic frameworks. More specifically, using a decrease threshold of $c\delta^q$, while reducing the number of samples needed to certify a step, it increases the iteration complexity. In fact, in the case of smooth objectives with stochastic oracles, an iteration complexity of $O(n\epsilon^{-q/(q-1)})$ for gradient norm, with $q \in (1,2]$, was proved in \cite{dzahini2025direct} for a direct-search scheme similar to the one given in~\cite{rinaldi2024stochastic}. 

It is finally important to highlight that all the methods mentioned above are fixed-sampling schemes, which always pay the worst-case cost. This basically means that, in order to satisfy the assumptions needed for convergence, one always needs to take the prescribed number of samples, no matter how obvious the decision is.
In~\cite{achddou2024stochastic}, the authors gave a sequential sampling strategy for a stochastic direct-search scheme for which a sequential hypothesis test has also been given. They claim a $O ((\log T)^{\frac{2}{3}} T^{\frac{2}{3}}  )$ regret bound with respect to a sample budget $T$ for smooth and strongly convex objectives, which would translate, neglecting the polylog term, into an overall complexity of $O(n\epsilon^{-6})$ for the gradient norm.

\subsection{Our contribution}\label{s12}

In Section~\ref{s2}, 
we introduce a new way of testing the satisfaction of a sufficient decrease condition in stochastic derivative-free optimization by framing it as a hypothesis test problem and solving it through the means of a sequential hypothesis test. The test makes a decision between two hypotheses, essentially corresponding to accepting or rejecting the sufficient decrease condition, outputting the probabilities of making correct and incorrect decisions. For the purpose of establishing a standard non-convex iteration complexity result, such probabilities need to satisfy certain bounds dependent on the algorithmic stepsize to ensure enough correctness. Specifically, this test sequentially collects noisy observations of the potential decrease, 
by calling a zeroth-order oracle (at the current and trial points) until their sum crosses either a lower or an upper bound depending on the function variance and the stepsize. When the function noise is Gaussian, we show that the size of the sample required to estimate the decrease drops significantly when the potential decrease is far from a multiple of the square of the stepsize, in which case we observe an early termination of the test. Furthermore, when the potential decrease is $\Theta(\delta^r)$ for some~$r\in(0,2]$, we show that the expected sample size decreases from the known $\Theta(\delta^{-4})$ to $O(\delta^{-2-r})$.

In Section~\ref{s3}, 
we apply this sequential test sampling framework to probabilistic-descent direct search when the function is stochastic. We first show that the expected decrease of an auxiliary merit function is sufficiently large when compared with the square of the stepsize and that the stepsize does not approach zero at non-stationary points. We then extend an existing renewal-reward supermartingale-based convergence rate analysis framework to a general case where the probability defining the Bernoulli process in~(\ref{pdefinition}) is arbitrary. Finally, we conclude that the iteration complexity of probabilistic-descent direct search is $O(n/\epsilon^2)$. In Section~\ref{s4}, our numerical results indicate the superiority of the sequential hypothesis test against a fixed sample test in evaluating the sufficient decrease condition when the function is stochastic. 

The use of our sequential test sampling framework, guarantees a similar iteration complexity as the deterministic case while guaranteeing, under the Gaussian noise assumption, a reduced  sample complexity with respect to the standard sample complexity obtained considering the  finite variance noise assumption. Furthermore, the sequential test sampling guarantees more flexibility than fixed-sampling schemes. In fixed-sample approaches, such as those considered in~\cite{rinaldi2024stochastic}, the same order of samples must indeed be taken at every iteration to satisfy the required convergence conditions, even when the trial step is clearly acceptable or clearly unacceptable, thus leading to potentially waste in terms of function evaluations. On the other side, the sequential test dynamically adapts the number of samples to the given scenario. Hence, when the decrease is clearly above or below the acceptance threshold, the test terminates after generating just a small number of samples. This way of adapting the number of generated samples gives a reduction of the expected per-iteration cost and provides a more efficient alternative to fixed-sampling strategies that always require, as mentioned before, the worst-case sample complexity.

\section{Sequential hypothesis testing framework for stochastic DFO}\label{s2}

A hypothesis test problem is a statistical framework used to decide between two competing hypotheses about a population based on observed samples. Two hypotheses are called the null hypothesis $H_0$ and the alternative hypothesis $H_1$. This problem is widely investigated in statistical inference.

Sequential hypothesis testing is a statistical method that uses sequential hypothesis tests to solve hypothesis test problems. It has a long history, which, according to its pioneer Wald~\cite{wald1992sequential}, may date back to the work~\cite{dodge1929method} of Dodge and Romig in 1929. Sigmund~\cite{siegmund2013sequential} points out that sequential hypothesis testing was developed to solve hypothesis test problems more efficiently, which also happens to be our goal.

In fact, many nonlinear optimization algorithms accept steps at a given iteration based on the satisfaction of a decrease condition on the value of the objective function. In derivative-free optimization, such a condition consists of imposing a simple decrease or a sufficient decrease on the objective function related to the size of a step. We are going to apply sequential hypothesis testing to enhance the satisfaction of the sufficient decrease condition when the function is stochastic.

\subsection{Testing a sufficient decrease condition}\label{s21}

Denote the current iterate of a DFO algorithm by $x$, the current stepsize by $\delta$, and the current candidate point by $x+\delta d$, where $d$ is a certain direction. Suppose that the candidate point $x+\delta d$ is accepted if the sufficient decrease condition
\begin{align}
    f(x)-f(x+\delta d)-c\delta^2\geq0\label{eq2}
\end{align}
is satisfied. Our goal is to reformulate its evaluation as a hypothesis test problem which can be used by the algorithm at stake.

For this purpose, denote by $F(x,\xi^x)$ one random observation of $f(x)$ and by $F(x+\delta d,\xi^d)$ one random observation of $f(x+\delta d)$. The notation $\xi^x$ and $\xi^d$ is used to clarify that $F(x,\xi^x)$ and $F(x+\delta d,\xi^d)$ are sampled independently. Let us define a random variable
\begin{align}\label{ydefinition}
    Y=c\delta^2-(F(x,\xi^x)-F(x+\delta d,\xi^d)).
\end{align}
We have from the unbiased noise assumption~(\ref{eq1}) that
\begin{align}
    E[Y]=c\delta^2-\left(f(x)-f(x+\delta d)\right).\label{eq3}
\end{align}
Since we do not know whether the sufficient decrease condition (\ref{eq2}) is satisfied, it follows from (\ref{eq3}) that two hypotheses to be considered in an algorithm are
\begin{align*}
    H_0:E[Y]\leq 0\\
    H_1:E[Y]> 0.
\end{align*}
Now it becomes clear that our hypothesis test problem is to decide whether the mean of a random variable is positive or not through observations. Since we are not directly interested in the value of $E[Y]$, our hypothesis test problem is different from a parameter estimation problem. A very accurate estimate of $E[Y]$ can be a burden when $E[Y]$ is far from 0. If the first few observations tell us that $E[Y]$ may be far from 0, then we may not want to obtain a very accurate estimate of it.

\subsection{The hypothesis test problem and sequential hypothesis tests}\label{s22}
To illustrate how to test the sufficient decrease condition in a sequential hypothesis testing framework, we first state our hypothesis test problem.

\begin{problem}[Hypothesis test problem]\label{Subp}
Let $Y$ be a random variable. The mean of the random variable $Y$ is denoted by $\mu\in\mathbb{R}$ and is unknown. The hypothesis test problem is to decide between two hypotheses
\begin{align*}
    H_0:\mu\leq0\\
    H_1:\mu>0
\end{align*}
on the basis of $m$ independent observations $Y^1,\ldots,Y^m$ drawn from $Y$, where $m$ is a random variable called stopping rule and defined on every sample $\omega=(Y^1,Y^2,\ldots)$.
\end{problem}

Then we give the definition of a sequential hypothesis test, which is adopted from~\cite{wald1948optimum}.

\begin{definition}[Sequential hypothesis test]
A sequential hypothesis test consists of a stopping rule $m(\omega)$ and a decision rule to decide $H_0$ or $H_1$ in a hypothesis test problem.
\end{definition}

A sequential hypothesis test is usually expected to end with a finite number of observations almost surely. We give the following definition with respect to this property.
\begin{definition}
    We say that a sequential hypothesis test ends properly if
\begin{align}
    P(m(\omega)<\infty)=1.\label{eq6}
\end{align}
\end{definition}

Finally, for Problem~\ref{Subp}, we give the following definition regarding the accuracy property of a given sequential hypothesis test.

\begin{definition}[$C$-accurate sequential hypothesis test]\label{Q}
For Problem~\ref{Subp}, we say that a sequential hypothesis test is $C$-accurate if its error probabilities satisfy
\begin{align}
    P(\mbox{$H_1$ is accepted}\,|\,\mu\leq0)&\leq \frac{1}{2}\label{eq7}\\
    P(\mbox{$H_0$ is accepted}\,|\,\mu> 0)&\leq\frac{C}{\mu}.\label{eq8}
\end{align}
\end{definition}

Two important properties of a sequential hypothesis test are its probability of accepting each hypothesis and its expected number of observations used or expected sample size. For $j=0,1$, we define that the acceptance region $S_j=\{\omega:\mbox{$H_j$ is accepted}\}$ of a sequential hypothesis test is the sample set where $H_j$ is accepted. We denote the expected sample size by $E_\mu[m]=E[m|\mu]$ and the probability of accepting $H_j$ by $P_\mu(S_j)=P(S_j|\mu)$ for $j=0,1$, when the mean of $Y$ has the value of $\mu$. If $\mu>0$, then $P_\mu(S_0)$ is the error probability in (\ref{eq8}). Similarly $P_\mu(S_1)$ is the error probability in (\ref{eq7}) when $\mu\leq0$.

A $C$-accurate sequential hypothesis test draws inferences of the mean of a random variable and delivers accuracy conditions (\ref{eq7}) and (\ref{eq8}) for its error probabilities. Conditions (\ref{eq7}) and (\ref{eq8}) represent a certain level of accuracy requirement for the error of the solution of Problem \ref{Subp} and will be used in Section~\ref{s3} to prove a convergence rate or complexity result.

\subsection{The proposed sequential hypothesis test}\label{s23}
Now we propose a sequential hypothesis test (see Test~\ref{ST}) for Problem~\ref{Subp} and study its properties.
\begin{sht}[A sequential hypothesis test]\label{ST}
\begin{algorithm}[H]
\begin{algorithmic}[]
  \STATE
  \STATE Specify $\{a_l\}$ and $\{b_l\}$ such that $a_l,b_l\in[-\infty,\infty]$, and $a_l\geq b_l$ for each $l$.
  \STATE \textbf{Repeat} for $l=1,2,\ldots$
  \STATE \quad Draw a new~i.i.d.~observation $Y^l$ from $Y$.
  \STATE \textbf{Until} $\sum_{i=1}^{l}Y^i\geq a_l$ or $\sum_{i=1}^{l}Y^i\leq b_l$.
  \STATE Record the number of used samples with $m(\omega)=l$.
  \STATE Decide that $H_0$ is true if $\sum_{i=1}^{m}Y^i\leq b_m$.
  \STATE Decide that $H_1$ is true if $\sum_{i=1}^{m}Y^i\geq a_m$.
\end{algorithmic}
\end{algorithm}
\end{sht}
Therefore, Test~\ref{ST} continues to draw observations from $Y$ until one of the two termination conditions is satisfied. To check whether Test~\ref{ST} fits the definition of a sequential hypothesis test, we notice that the stopping rule in Test~\ref{ST} is
\begin{align*}
    m(\omega)=\inf\{l\geq1:\sum_{i=1}^{l}Y^i\geq a_l\ \mbox{or} \sum_{i=1}^{l}Y^i\leq b_l\}.
\end{align*}
Then Test~\ref{ST} decides that $H_0$ or $H_1$ is accepted based on which termination condition is satisfied. Specifically, it decides that $H_0$ is true if $\sum_{i=1}^{m}Y_i$ is smaller than $b_m$ and that $H_1$ is true if $\sum_{i=1}^{m}Y_i$ is larger than $a_m$.

To see the generality of Test~\ref{ST}, we notice that a sampling procedure with a predetermined fixed sample size at each iteration, which is what most stochastic DFO algorithms employ, is also an instance of Test~\ref{ST}. We describe such a sampling procedure in the following test.
\begin{sht}[A test with fixed sample size $m$]\label{FT}
\begin{algorithm}[H]
\begin{algorithmic}[]
  \STATE
  \STATE Draw $m$~i.i.d.~observation $Y^1,Y^2,\ldots,Y^m$ from $Y$.
  \STATE Decide that $H_0$ is true if $\sum_{i=1}^{m}Y^i\leq 0$.
  \STATE Decide that $H_1$ is true if $\sum_{i=1}^{m}Y^i>0$.
\end{algorithmic}
\end{algorithm}
\end{sht}
We can easily check that Test~\ref{FT} is an instance of Test~\ref{ST}, by choosing the parameters in Test~\ref{FT} as $a_l=\infty$ and $b_l=-\infty$ for $l<m$ and $a_m=b_m=0$. Test~\ref{FT} is usually referred to as a fixed sample size hypothesis test in the literature of statistics. It can be shown through Markov's inequality that Test~\ref{FT} delivers the accuracy condition~(\ref{eq8}) when at least $m=\sigma^2C^{-2}$ samples are used. When $C$ is $\Theta(\delta^2)$, which is commonly required in stochastic DFO, Test~\ref{FT} requires a sample set of size~$\Theta(\delta^{-4})$, which is known to be what most stochastic DFO algorithms need per iteration for standard convergence and convergence rates. This may further explain that most stochastic DFO algorithms using Test~\ref{FT} require a sample size of $\Theta(\delta^{-4})$ per iteration.

To make sure that condition~(\ref{eq6}) holds and Test~\ref{ST} ends properly, we need to choose $\{a_l\}$ and $\{b_l\}$ such that
\begin{align}
    P(\{b_l\leq\sum_{i=1}^{l}Y_i\leq a_l,\ \forall l\})=0.\label{eq!}
\end{align}
Notice that $\sum_{i=1}^{l}Y_i$ is a one-dimensional random walk with~i.i.d.~increments. If $\{a_l\}$ and $\{b_l\}$ are two bounded sequences and $Y$ is not almost surely zero, then condition (\ref{eq!}) holds, and hence also~(\ref{eq6}).

Test~\ref{ST} is $C$-accurate if we choose the appropriate parameters $\{a_l\}$ and $\{b_l\}$ such that conditions (\ref{eq7}) and (\ref{eq8}) hold. To investigate when condition (\ref{eq7}) holds, the following lemma tells us that $P_\mu(S_j)$ of any Test~\ref{ST} is monotone with respect to $\mu$. The proof consists of two steps. In the first step, given any $\Delta\mu>0$ and any sequential test Test~\ref{ST}, we define an ancillary sequential hypothesis test, called~Test~\ref{ST}*, such that $P_{\mu+\Delta\mu}(S_j)=P_{\mu}(\Tilde{S}_j)$, where $S_j$ and $\Tilde{S}_j$ are the acceptance regions of $H_j$ in Test~\ref{ST} and Test~\ref{ST}* accordingly. In the second step, we note that Test~\ref{ST}* makes each sample harder for $H_0$ and easier for $H_1$ than Test~\ref{ST}, and therefore, that $P_\mu(\Tilde{S}_0)\leq P_\mu(S_0)$ and $P_\mu(\Tilde{S}_1)\geq P_\mu(S_1)$.

\begin{lemma}\label{L21}
    Consider Problem \ref{Subp} and a sequential test Test~\ref{ST}, whose acceptance regions are $S_j$, for $j=0,1$. Then $P_\mu(S_0)$ is non-increasing with respect to $\mu$ and $P_\mu(S_1)$ is non-decreasing with respect to $\mu$.
\end{lemma}

\begin{proof}
    It suffices to prove $P_\mu(S_0)\geq P_{\mu+\Delta\mu}(S_0)$ and $P_\mu(S_1)\leq P_{\mu+\Delta\mu}(S_1)$ for any $\Delta\mu>0$. For the purpose of the proof, we define an ancillary sequential hypothesis test in the form of Test~\ref{ST}, called Test~\ref{ST}*, by selecting parameters $\{a_l-l\Delta\mu\}$ and $\{b_l-l\Delta\mu\}$, where $\{a_l\}$ and $\{b_l\}$ are the parameters of the given Test~\ref{ST}. Denote its acceptance region by $\Tilde{S}_j$ and its sample size by $\Tilde{m}$.

We first prove that $P_{\mu+\Delta\mu}(S_j)=P_{\mu}(\Tilde{S}_j)$. For each $\omega=(y_1,y_2,\ldots)$, we make a change of variable and define $\Tilde{\omega}=(\Tilde{y}_1,\Tilde{y}_2,\ldots)=(y_1-\Delta\mu,y_2-\Delta\mu,\ldots)$. From the definition of Test~\ref{ST}*, it follows that $m(\omega)=\Tilde{m}(\Tilde{\omega})$ and $\omega\in S_j$ if and only if $\Tilde{\omega}\in\Tilde{S}_j$. It also follows from this change of variable that $P_{\mu+\Delta\mu}(S_j)=P_{\mu}(\Tilde{S}_j)$.

It then suffices to prove that $P_\mu(\Tilde{S}_0)\leq P_\mu(S_0)$ and $P_\mu(\Tilde{S}_1)\geq P_\mu(S_1)$. From the definition of $S_0$ and $\Tilde{S}_0$, we have $\Tilde{S}_0\subseteq S_0$ and $P_\mu(\Tilde{S}_0)\leq P_\mu(S_0)$. Similarly, we have $\Tilde{S}_1\supseteq S_1$ and $P_\mu(\Tilde{S}_1)\geq P_\mu(S_1)$.
\end{proof}

Lemma \ref{L21} tells us that $P_\mu(S_1)$ is non-decreasing, which implies that the left-hand side of~(\ref{eq7}) is non-decreasing and admits its maximum when $\mu=0$. Hence to satisfy condition~(\ref{eq7}), it suffices to make sure that the maximum of the left-hand side of (\ref{eq7}) is no larger than $1/2$. We formalize this reasoning in the following lemma.

\begin{lemma}\label{L22}
    For any sequential test Test~\ref{ST}, one has that (\ref{eq7}) holds if and only if
\begin{align}
    P(\mbox{$H_1$ is accepted}\,|\,\mu=0)&\leq \frac{1}{2}.\label{CL2}
\end{align}
\end{lemma}

\begin{proof}
    We have, by definition, $P_\mu(S_1)=P(S_1\,|\,\mu)=P(\mbox{$H_1$ is accepted}\,|\,\mu)$. Then it follows from Lemma 1 that $P(\mbox{$H_1$ is accepted}\,|\,\mu)$ is non-decreasing with respect to $\mu$
\begin{align*}
    P(\mbox{$H_1$ is accepted}\,|\,\mu\leq0)\leq P(\mbox{$H_1$ is accepted}\,|\,\mu=0).
\end{align*}
So, as a function of $\mu$, $P(\mbox{$H_1$ is accepted}\,|\,\mu\leq0)$ reaches its maximum when $\mu=0$. Therefore,~(\ref{eq7}) holds if and only if $P(\mbox{$H_1$ is accepted}\,|\,\mu=0)\leq 1/2$.
\end{proof}

Lemma \ref{L22} gives us an equivalent condition (\ref{CL2}) for (\ref{eq7}) when Test~\ref{ST} is used. One way of ensuring (\ref{eq7}) through the satisfaction of (\ref{CL2}) is when the probability density function~$\phi_\mu(y)$ of $Y$ is symmetric and $a_l+b_l=0$ for each $l>0$.

\begin{lemma}\label{L23}
    Assume that $\phi_\mu(\mu+y)=\phi_\mu(\mu-y)$ for any $y$. Select $a_l$ and $b_l$ in Test~\ref{ST} such that $a_l+b_l=0$ holds for each $l>0$. Then (\ref{eq7}) holds.
\end{lemma}

\begin{proof}
    The symmetry of the probability density function $\phi_\mu(y)$ and $a_l+b_l=0$ implies $P(\mbox{$H_1$ is accepted}\,|\,\mu=0)=1/2$. Then (\ref{CL2}) holds and condition~(\ref{eq7}) follows from Lemma \ref{L22}.
\end{proof}

Now we focus our attention on the satisfaction of condition (\ref{eq8}). When we use Test~\ref{ST} in the context of stochastic DFO algorithms, the value of $C$ in (\ref{eq8}) can be small, and so we need to investigate how large a sample size needs to be for the satisfaction of~(\ref{eq8}). We will show in the next subsection that, if $Y$ in Problem~\ref{Subp} follows a Gaussian distribution, we can select parameters in Test~\ref{ST} so that condition (\ref{eq8}) holds.

\subsection{Sample size in the Gaussian case}\label{s24}
In this subsection, we assume that $Y$ in Problem~\ref{Subp} follows a Gaussian distribution. Let $\phi_\mu(y)=(2\pi\sigma^2)^{-1/2}\exp{(-(y-\mu)^2/(2\sigma^2))}$ be the probability density function of $Y$. We notice that, when~$Y$ is Gaussian, Test~\ref{ST} with constant parameters is equivalent to a sequential probability ratio test in~\cite{wald1992sequential}. We will specify the parameters in Test~\ref{ST} so that Test~\ref{ST} is $C$-accurate. Then we approximate its expected number of observations $E_\mu[m]$.

We first select constant parameters in Test~\ref{ST} as $a_l=-b_l=c^0$ for each $l$, where $c^0>0$ is a real number. Lemma~\ref{L23} shows that condition~(\ref{eq7}) is valid. Furthermore, in the Gaussian case, Test~\ref{ST} is equivalent to a sequential probability ratio test in~\cite{wald1992sequential}, which has four real number parameters $A>1$, $0<B<1$, $\theta_0$, and $\theta_1>\theta_0$, given that our parameter satisfies $c^0=\sigma^2\log{A}/(\theta_1-\theta_0)$ and $-c^0=\sigma^2\log{B}/(\theta_1-\theta_0)$.

For this sequential probability ratio test, Wald bounded the left-hand side of condition (\ref{eq8}) in \cite[(3.42)]{wald1992sequential}, and it follows that
\begin{align*}
    P(\mbox{$H_0$ is accepted}\,|\,\mu> 0)\leq A^{-h},
\end{align*}
where $h$ denotes $2\mu/(\theta_1-\theta_0)$. We can use his result for our Test~\ref{Subp} because of the equivalence between the two tests. Since we have equivalently $c^0=\sigma^2\log{A}/(\theta_1-\theta_0)$, we can express $A$ with respect to $c^0$ as $A=\exp{(c^0(\theta_1-\theta_0)/\sigma^2)}$. After substituting $h$ and $A$, we have for Test~\ref{ST} that
\begin{align*}
    P(\mbox{$H_0$ is accepted}\,|\,\mu> 0)\leq e^{-\frac{2c^0}{\sigma^2}\mu}.
\end{align*}
It then follows from Proposition 2 in Appendix A with $x=\mu/C$ that, given $C>0$, we can select $c^0\geq\sigma^2/(2eC)$ so that for any $\mu>C$, we have $\exp(-2c^0\mu/\sigma^2)\leq C/\mu$. Therefore, by setting parameters~$a_l=-b_l=c^0\geq\sigma^2/(2eC)$ in Test~\ref{ST}, we can ensure that condition~(\ref{eq8}) holds and Test~\ref{ST} is $C$-accurate in the Gaussian case.

Calculating the exact value of the expected sample size is more elaborate. Wald~\cite{wald2004sequential} gave an approximate expected sample size in the Gaussian case. For parameters $A=\exp{(c^0(\theta_1-\theta_0)/\sigma^2)}$, $B=1/A$, and $c^0=\sigma^2/(2eC)$ in~\cite[(3:43)]{wald2004sequential} and in~\cite[(3:57)]{wald2004sequential}, the number is as follows
\begin{align}
    E_{\mu}[m]\approx\frac{\sigma^2}{2eC\mu}\frac{e^{\frac{\mu}{eC}}-1}{e^{\frac{\mu}{eC}}+1}.\label{eq9}
\end{align}
To help clarify it, it follows from Proposition 1 in Appendix A with $x=\mu/(eC)$ that we can bound (\ref{eq9}) as follows
\begin{align}
    \frac{\sigma^2}{2eC\mu}\frac{e^{\frac{\mu}{eC}}-1}{e^{\frac{\mu}{eC}}+1}\leq\frac{\sigma^2}{4e^2C^2}\min\left(1,\frac{eC}{|\mu|}\right).\label{eq10}
\end{align}

In the context of stochastic DFO algorithms, the level of accuracy is $C=s\delta^2$ (for some constant $s>0$) and $\mu=c\delta^2-\left(f(x)-f(x+\delta d)\right)$. After we substitute these values of $C$ and $\mu$ into~(\ref{eq10}), we obtain the following proposition.

\begin{proposition}
\label{prop}When evaluating the sufficient decrease condition (\ref{eq2}) in a stochastic DFO algorithm using Test~\ref{ST}, if $Y=c\delta^2-(F(x,\xi^x)-F(x+\delta d,\xi^d))$ is Gaussian with variance~$\sigma^2$ and $a_l=-b_l=\sigma^2/(2es\delta^2)$, then one needs an expected sample size approximately bounded by
\begin{align}
    \frac{\sigma^2}{4e^2s^2}\delta^{-4}\min\left(1,\frac{es\delta^2}{|c\delta^2-(f(x)-f(x+\delta d))|}\right)\label{eq27}
\end{align}
so that Test~\ref{ST} is $s\delta^2$-accurate.
\end{proposition}

Proposition~\ref{prop} gives us a novel sample size result, which explicitly depends on the potential decrease $f(x)-f(x+\delta d)$. Here we offer some interpretation and insight on the novel sample size quantity~(\ref{eq27}). Since $\min(a,b)$ is smaller than both  the first term $a$ and the second term $b$, it is clear from the first term that this quantity is $O(\delta^{-4})$. Therefore, the sample size in Test~\ref{ST} is at least not larger than in Test~\ref{FT} with respect to the power of $\delta$. From the second term, the quantity~(\ref{eq27}) is $O(\delta^{-2}/|c\delta^2-(f(x)-f(x+\delta d))|)$. In fact, when the potential decrease $f(x)-f(x+\delta d)$ is $\Theta(\delta)$, the quantity~(\ref{eq27}) drops to $O(\delta^{-3})$. Combining both terms, on the one hand, the quantity~(\ref{eq27}) achieves its maximum when~$f(x)-f(x+\delta d)$ is close to~$c\delta^2$, where it becomes difficult to distinguish the true hypothesis. On the other hand, Test~\ref{ST} stops early and this quantity decreases significantly whenever~$f(x)-f(x+\delta d)$ is far from $c\delta^2$, in which case the ratio~$es\delta^2/|c\delta^2-(f(x)-f(x+\delta d))|$ becomes small. This happens frequently when the algorithm is far from stationarity and~$f(x)-f(x+\delta d)$ is~$\Theta(\delta)$. Such an early termination effect and sample size advantage are commonly seen in the sequential hypothesis test theory~\cite{wald2004sequential,siegmund2013sequential}.

\section{Convergence rate of probabilistic-descent direct search}\label{s3}
In this section, to handle the stochastic setting, we propose a version of probabilistic-descent direct search in which its sufficient decrease condition is tested in a sequential hypothesis testing framework, through a sequential hypothesis test that ends properly and is $C$-accurate, where $C>0$ is given at the beginning of each algorithm iteration. Before we introduce the algorithm, we formally state the assumption that such a sequential hypothesis test can be developed for any given $C>0$.

\begin{assumption}\label{A1}
    For any $C>0$, there exists a $C$-accurate sequential hypothesis test for Problem~\ref{Subp} that ends properly.
\end{assumption}
For the rest of this section, we will assume that Assumption~\ref{A1} is satisfied. Probabilistic-descent direct search in stochastic setting is stated in Algorithm~\ref{SDS}. The value of $C$ is chosen as a multiple of $\Delta_k^2$, where $\Delta_k$ is the stepsize, so $C$ varies in each iteration and becomes small when $\Delta_k$ does. A $C_k$-accurate sequential hypothesis test will determine whether a candidate point is accepted (and the stepsize increased) or rejected (and the stepsize decreased).
\begin{algorithm}[H]
\caption{Probabilistic-descent direct search based on sequential hypothesis testing}\label{SDS}
\begin{algorithmic}[1]
\STATE{Initialization. Choose an initial point $x_0$, an initial stepsize $\delta_0$, $c>0$, $\theta\in(0,1)$, $\gamma\in(1,\infty)$.}
\FOR{$k=0,1,\cdots$}
\STATE{Uniformly select a random direction $D_k$ from the unit sphere.}
\STATE{The candidate point is then $X_k+\Delta_kD_k$. Perform a $c\Delta_k^2(1-\theta^2)/2(\gamma^2-\theta^2)$-accurate sequential hypothesis test for Problem~\ref{Subp} with the random variable $Y_k=c\Delta_k^2-(F(X_k,\xi^x)-F(X_k+\Delta_k D_k,\xi^d))$ (which requires estimating~$\sigma$). The result of this test declares either $H_0$ or $H_1$ accepted.}
\IF{$H_0$ is accepted,}
\STATE{The candidate point is accepted. Set $X_{k+1}=X_k+\Delta_kD_k$ and $\Delta_{k+1}=\gamma\Delta_{k}$.}
\ELSE
\STATE{The candidate point is rejected. Set $X_{k+1}=X_k$ and $\Delta_{k+1}=\theta\Delta_{k}$.}
\ENDIF
\ENDFOR
\end{algorithmic}
\end{algorithm}

We need to introduce some notation to develop a convergence rate for Algorithm \ref{SDS}. Let us denote the decision result of the sequential hypothesis test at iteration $k$ by
\begin{align*}
    A_k=\mathbf{1}\{\mbox{$H_0$ is accepted at iteration $k$}\}.
\end{align*}
In Algorithm \ref{SDS}, the directions $D_k$ and the decisions $A_k$ are random. Denote the probability space of Algorithm \ref{SDS} by $(\Omega,\mathcal{F},P)$. As a consequence of the randomness of $D_{k-1}$ and $A_{k-1}$, the current point $X_k$ and the stepsize $\Delta_k$ are also random quantities. Let $\Phi_k=f(X_k)-f^*+\eta\Delta_k^2$, where $\eta>0$ is a real number. To formalize conditioning on the past, let $\mathcal{F}_{k}$ denote the $\sigma$-algebra generated by $D_0,\ldots,D_{k-1}$ and $A_0,\ldots,A_{k-1}$ and let $\mathcal{F}_{k+1/2}$ denote the $\sigma$-algebra generated by $D_0,\ldots,D_k$ and $A_0,\ldots,A_{k-1}$. Under this definition of $\mathcal{F}_{k}$ and $\mathcal{F}_{k+1/2}$, we have the following conclusions: $\Delta_k$, $X_k$, and $\Phi_k$ are $\mathcal{F}_{k}$-measurable; $\Delta_k$, $X_k$, $\Phi_k$, and $D_k$ are $\mathcal{F}_{k+1/2}$-measurable; $\Delta_k$, $X_k$, $\Phi_k$, $D_k$, and $A_k$ are $\mathcal{F}_{k+1}$-measurable.

To analyze Algorithm \ref{SDS}, we start by noting that the unbiased noise Assumption~(\ref{eq1}) implies
\begin{align}
    E[Y_k|\mathcal{F}_{k+1/2}]=c\Delta_k^2-(f(X_k)-f(X_k+\Delta_k D_k)),\label{eq30}
\end{align}
where $Y_k$ was defined in Step 4 of Algorithm \ref{SDS}. Denote $S^1_k=\{\omega:f(X_k)-f(X_k+\Delta_k D_k)\geq c\Delta_k^2\}$ and $S^2_k=\{\omega:f(X_k)-f(X_k+\Delta_k D_k)< c\Delta_k^2\}$, which are two disjoint $\mathcal{F}_{k+1/2}$-measurable events. It follows from~(\ref{eq30}) that $S^1_k=\{\omega:E[Y_k]\leq0\}$ and $S^2_k=\{\omega:E[Y_k]>0\}$. We can then rewrite (\ref{eq7}) and (\ref{eq8}) in the context of Algorithm \ref{SDS} in the following way
\begin{align}
    P(A_k=0|\mathcal{F}_{k+1/2},S_k^1)&\leq \frac{1}{2}\label{eq31}\\
    P(A_k=1|\mathcal{F}_{k+1/2},S_k^2)&\leq\frac{c\Delta_k^2(1-\theta^2)/2(\gamma^2-\theta^2)}{c\Delta_k^2-(f(X)-f(X_k+\Delta_k D_k))}.\label{eq32}
\end{align}
Note that $E[Y_k]$ plays here the role of $\mu$ in (\ref{eq7}) and (\ref{eq8}).

Our main goal is to bound the number of iterations $T_\epsilon$ defined as follows
\begin{align}\label{T}
    T_\epsilon=\inf\{k\geq0:\|\nabla f(X_k)\|\leq\epsilon\}.
\end{align}
The first lemma tells us that the expected value decrease of $\Phi_k$ is bounded below by $\nu\Delta_k^2$, where $\nu>0$ is a real number.

\begin{lemma}\label{3L1}
    Let Assumption~\ref{A1} hold and $\Phi_k=f(X_k)-f^*+\eta\Delta_k^2$. Then there exist $\eta=\frac{c}{\gamma^2-\theta^2}>0$ and $\nu=\frac{c}{2}\frac{1-\theta^2}{\gamma^2-\theta^2}$ such that for all $k$ we have
\begin{align*}
    E[\Phi_k-\Phi_{k+1}|\mathcal{F}_{k}]\geq\nu\Delta^2_k.
\end{align*}
\end{lemma}

\begin{proof}
    We start by proving $E[\Phi_k-\Phi_{k+1}|\mathcal{F}_{k+1/2}]\geq\nu\Delta^2_k$. Note that $\Delta_k$, $X_k$, and $D_k$ are $\mathcal{F}_{k+1/2}$-measurable. We express $E[\Phi_k-\Phi_{k+1}|\mathcal{F}_{k+1/2}]$ in terms of $P(A_k=1|\mathcal{F}_{k+1/2})$ as follows
\begin{align*}
    E[\Phi_k-\Phi_{k+1}|\mathcal{F}_{k+1/2}]\ =\ &E[f(X_k)+\eta\Delta_k^2-f(X_{k+1})-\eta\Delta_{k+1}^2|\mathcal{F}_{k+1/2}]\\
    =\ &E[\left(f(X_k)+\eta\Delta_k^2-f(X_{k+1})-\eta\Delta_{k+1}^2\right)\mathbf{1}(A_k=1)|\mathcal{F}_{k+1/2}]\\
    &+E[\left(f(X_k)+\eta\Delta_k^2-f(X_{k+1})-\eta\Delta_{k+1}^2\right)\mathbf{1}(A_k=0)|\mathcal{F}_{k+1/2}]\\
    =\ &(f(X_k)-f(X_k+\Delta_kD_k)+\eta(1-\gamma^2)\Delta_k^2)P(A_k=1|\mathcal{F}_{k+1/2})\\
    &+\eta(1-\theta^2)\Delta_k^2P(A_k=0|\mathcal{F}_{k+1/2}).
\end{align*}
Since $P(A_k=0|\mathcal{F}_{k+1/2})=1-P(A_k=1|\mathcal{F}_{k+1/2})$, we have
\begin{align*}
    E[\Phi_k-\Phi_{k+1}|\mathcal{F}_{k+1/2}]=\ &(f(X_k)-f(X_k+\Delta_kD_k)+\eta(\theta^2-\gamma^2)\Delta_k^2)P(A_k=1|\mathcal{F}_{k+1/2})\\
    &+\eta(1-\theta^2)\Delta_k^2\\
    =\ &(f(X_k)-f(X_k+\Delta_kD_k)-c\Delta_k^2)P(A_k=1|\mathcal{F}_{k+1/2})\\
    &+c\frac{1-\theta^2}{\gamma^2-\theta^2}\Delta_k^2.
\end{align*}

Note that $S^1_k=\{\omega:f(X_k)-f(X_k+\Delta_k D_k)\geq c\Delta_k^2\}$ and $S^2_k=\{\omega:f(X_k)-f(X_k+\Delta_k D_k)< c\Delta_k^2\}$ are two disjoint $\mathcal{F}_{k+1/2}$-measurable events such that $S^1_k\cup S^2_k=\Omega$. We prove $E[\Phi_k-\Phi_{k+1}|\mathcal{F}_{k+1/2}]\geq\nu\Delta^2_k$ for these two scenarios separately.

\textbf{Case 1.} First we consider the case $S_k^1=\{\omega:f(X_k)-f(X_k+\Delta_k D_k)\geq c\Delta_k^2\}$ and it follows from $P(A_k=1|\mathcal{F}_{k+1/2})\geq0$ that
\begin{align*}
    E[\Phi_k-\Phi_{k+1}|\mathcal{F}_{k+1/2},S_k^1]\geq c\frac{1-\theta^2}{\gamma^2-\theta^2}\Delta_k^2\geq\nu\Delta^2_k.
\end{align*}

\textbf{Case 2.} Otherwise, we consider $S_k^2=\{\omega:f(X_k)-f(X_k+\Delta_kD_k)<c\Delta_k^2\}$. Then (\ref{eq32}) guarantees
\begin{align*}
    P(A_k=1|\mathcal{F}_{k+1/2},S_k^2)\leq\frac{c\Delta_k^2(1-\theta^2)/2(\gamma^2-\theta^2)}{c\Delta_k^2-(f(X_k)-f(X_k+\Delta_k D_k))}.
\end{align*}
It follows that
\begin{align*}
    E[\Phi_k-\Phi_{k+1}|\mathcal{F}_{k+1/2},S_k^2]=&(f(X_k)-f(X_k+\Delta_kD_k)-c\Delta_k^2)P(A_k=1|\mathcal{F}_{k+1/2},S_k^2)+c\frac{1-\theta^2}{\gamma^2-\theta^2}\Delta_k^2\\
    \geq&(c\frac{1-\theta^2}{\gamma^2-\theta^2}-\frac{c}{2}\frac{1-\theta^2}{\gamma^2-\theta^2})\Delta_k^2=\nu\Delta_k^2.
\end{align*}
Since $S^1_k$ and $S^2_k$ partition $\Omega$ and are $\mathcal{F}_{k+1/2}$-measurable, we conclude that
\begin{align*}
    E[\Phi_k-\Phi_{k+1}|\mathcal{F}_{k+1/2}]\geq \nu\Delta^2_k.
\end{align*}

Note that $\Delta_k$ is $\mathcal{F}_{k}$-measurable. It follows from the tower property of conditional expectation that
\begin{align*}
    E[\Phi_k-\Phi_{k+1}|\mathcal{F}_{k}]&=E[E[\Phi_k-\Phi_{k+1}|\mathcal{F}_{k+1/2}]|\mathcal{F}_{k}]\\
    &\geq E[\nu\Delta^2_k|\mathcal{F}_{k}]=\nu\Delta^2_k.
\end{align*}
\end{proof}

Now we need to guarantee that the directions used in Algorithm \ref{SDS} are of descent type with a sufficiently large probability. To do so, we need to generate a direction for which
\begin{align*}
    \kappa_k=\frac{-\nabla f(X_k)^\top D_k}{\|\nabla f(X_k)\|\|D_k\|}
\end{align*}
is sufficiently large. Note that the denominator of $\kappa_k$ will not be zero because we will assume that $k<T_\epsilon$, where $T_\epsilon$ was defined in~(\ref{T}). From~\cite[Lemma B.2]{gratton2015direct}, it turns out that if $\mathcal{D}_k=\{D_k\}$ is uniformly generated from a unit sphere, the direction $D_k$ will be enough descent ($\kappa_k\geq1/(7\sqrt{n})$) for a sufficiently large probability.

\begin{lemma}
For $k<T_\epsilon$, it holds for $\tau\in[0,\sqrt{n}]$ that\label{3L2}
\begin{align*}
    P(\kappa_k\geq\frac{\tau}{\sqrt{n}}|\mathcal{F}_k)\geq\frac{1}{2}-\frac{\tau}{\sqrt{2\pi}}.
\end{align*}
For simplicity we select $\tau=\frac{1}{7}$ and it holds for $k<T_\epsilon$ that
\begin{align*}
    P(\kappa_k\geq\frac{1}{7\sqrt{n}}|\mathcal{F}_k)\geq\frac{3}{7}.
\end{align*}
\end{lemma}

\begin{proof}
    See~\cite[Lemma B.2]{gratton2015direct}, where the result holds for any $k$.
\end{proof}

The next step in the convergence theory is to ensure that for each $k<T_\epsilon$, if $D_k$ is an at least~$1/(7\sqrt{n})$-descent direction and the stepsize $\Delta_k$ is smaller than a certain $\delta_\epsilon$, then the sufficient decrease condition is satisfied at iteration $k$. To do so, we need another assumption on the objective function, which is common in the DFO literature.

\begin{assumption}\label{A2}
    Suppose that the objective function $f$ is bounded from below on $\mathbb{R}^n$. Suppose that $f$ is smooth and its gradient $\nabla f$ is Lipschitz continuous with constant $L_f$.
\end{assumption}

The next lemma is a consequence of assuming that $f$ is smooth and its gradient $\nabla f$ is $L_f$-Lipschitz continuous.

\begin{lemma}\label{3L3}
    Let Assumption \ref{3L2} hold. For all $k<T_\epsilon$, if $\kappa_k\geq\frac{1}{7\sqrt{n}}$ and $\Delta_k\leq\delta_\epsilon=\frac{2}{7L_f+14c}\frac{\epsilon}{\sqrt{n}}$, then
    \begin{align}\label{eql33}
        f(X_k)-f(X_k+\Delta_kD_k)\geq c\Delta_k^2.
    \end{align}
\end{lemma}

\begin{proof}
    It is known~(see, for example,~\cite[Lemma 1.2.3]{nesterov2013introductory}) that Assumption~\ref{3L2} implies for any $x,y$ from $\mathbb{R}^n$ that
    \begin{align*}
        f(x)-f(y)\geq\nabla f(x)^T(x-y)-\frac{L_f}{2}\|x-y\|^2.
    \end{align*}
    After substituting $x=X_k$ and $y=X_k+\Delta_kD_k$, and using $\|D_k\|=1$, we obtain
\begin{align}
    f(X_k)-f(X_k+\Delta_kD_k)&\geq-\nabla f(X_k)^\top (\Delta_kD_k)-\frac{L_f}{2}\Delta_k^2\notag\\
    &=\kappa_k\|\nabla f(X_k)\|\Delta_k-\frac{L_f}{2}\Delta_k^2.\label{eq23}
\end{align}
It then follows from $\kappa_k\geq\frac{1}{7\sqrt{n}}$, $\|\nabla f(X_k)\|>\epsilon$, $\Delta_k\leq\delta_\epsilon$, and the definition of $\delta_\epsilon$ that
\begin{align}
    \kappa_k\|\nabla f(X_k)\|\Delta_k\geq\frac{\epsilon}{7\sqrt{n}}\Delta_k\geq c\Delta_k^{2}+\frac{L_f}{2}\Delta_k^2.\label{eq24}
\end{align}
Combining (\ref{eq23}) with (\ref{eq24}) gives us~(\ref{eql33}).
\end{proof}

Lemma~\ref{3L3} gives a sufficient condition for the sufficient decrease condition to be satisfied. The next lemma gives a constant lower bound on the probability of accepting $H_0$ when the stepsize $\Delta_k$ is smaller than $\delta_\epsilon$. The proof relies on both the probability of having an at least $1/(7\sqrt{n})$-descent direction (Lemma~\ref{3L2}) and the conditional probability of correctly identifying sufficient decrease (given by condition~(\ref{eq31})).

\begin{lemma}\label{3L4}
    Let Assumptions \ref{A1}--\ref{A2} hold. For any $\epsilon>0$ and $k<T_\epsilon$, it holds
    \begin{align}
        P(A_k=1|\mathcal{F}_{k},\{\Delta_k\leq\delta_\epsilon\})\geq\frac{3}{14}.\label{34}
    \end{align}
\end{lemma}

\begin{proof}
    Note that $\Delta_k$ and $X_k$ are $\mathcal{F}_{k}$-measurable. It follows from event inclusion that
\begin{align*}
    P(A_k=1|\mathcal{F}_{k},\{\Delta_k\leq\delta_\epsilon\})\ \geq\ P(\{A_k=1\}\cap\{\kappa_k\geq\frac{1}{7\sqrt{n}}\}|\mathcal{F}_{k},\{\Delta_k\leq\delta_\epsilon\}).
\end{align*}
Then it follows from conditional probability formula $P(E_1E_2|E_3)=P(E_1|E_2E_3)P(E_2|E_3)$ that
\begin{align}
    P(A_k=1|\mathcal{F}_{k}&,\{\Delta_k\leq\delta_\epsilon\})\ \geq\ P(\{A_k=1\}\cap\{\kappa_k\geq\frac{1}{7\sqrt{n}}\}|\mathcal{F}_{k},\{\Delta_k\leq\delta_\epsilon\})\notag\\
    =&P(A_k=1|\mathcal{F}_{k},\{\kappa_k\geq\frac{1}{7\sqrt{n}}\}\cap\{\Delta_k\leq\delta_\epsilon\})\cdot P(\kappa_k\geq\frac{1}{7\sqrt{n}}|\mathcal{F}_{k},\{\Delta_k\leq\delta_\epsilon\}).\label{eq20}
\end{align}
From Lemma \ref{3L2}, since $\Delta_k$ is $\mathcal{F}_{k}$-measurable, we know that
\begin{align}
    P(\kappa_k\geq\frac{1}{7\sqrt{n}}|\mathcal{F}_{k},\{\Delta_k\leq\delta_\epsilon\})\geq\frac{3}{7}.\label{eq21}
\end{align}

From Lemma \ref{3L3} we have
\begin{align}
    \left\{\{\kappa_k\geq\frac{1}{7\sqrt{n}}\}\cap\{\Delta_k\leq\delta_\epsilon\}\right\}\subseteq\{f(X_k)-f(X_k+\Delta_kD_k)\geq c\Delta_k^2\}=S_k^1.\label{eqa}
\end{align}
From (\ref{eq31}) we have
\begin{align}
    P(A_k=1|\mathcal{F}_{k+1/2},S_k^1)>\frac{1}{2}.\label{eqb}
\end{align}
Together with the knowledge that both events in~(\ref{eqa}) above are $\mathcal{F}_{k+1/2}$-measurable,~(\ref{eqb}) implies
\begin{align*}
    P(A_k=1|\mathcal{F}_{k+1/2},\{\kappa_k\geq\frac{1}{7\sqrt{n}}\}\cap\{\Delta_k\leq\delta_\epsilon\})>\frac{1}{2}.
\end{align*}
Using the tower property of conditional expectation, it follows that
\begin{align}
    P(A_k=1&|\mathcal{F}_{k},\{\kappa_k\geq\frac{1}{7\sqrt{n}}\}\cap\{\Delta_k\leq\delta_\epsilon\})\notag\\
    &=E[P(A_k=1|\mathcal{F}_{k+1/2},\{\kappa_k\geq\frac{1}{7\sqrt{n}}\}\cap\{\Delta_k\leq\delta_\epsilon\})|\mathcal{F}_{k},\{\kappa_k\geq\frac{1}{7\sqrt{n}}\}\cap\{\Delta_k\leq\delta_\epsilon\}]\notag\\
    &>E[\frac{1}{2}|\mathcal{F}_{k},\{\kappa_k\geq\frac{1}{7\sqrt{n}}\}\cap\{\Delta_k\leq\delta_\epsilon\}]\notag\\
    &=\frac{1}{2}.\label{eq22}
\end{align}
Applying (\ref{eq21}) and (\ref{eq22}) to (\ref{eq20}) yields~(\ref{34}).
\end{proof}

The rest of the rate derivation makes use of the result~\cite[Theorem 2]{blanchet2019convergence} which is rederived in the Appendix for a general probability $p>1/2$ (see Theorem~\ref{BP}). This result requires defining a submartingale $W_k$ to model the behavior of the stepsize, which we now do in the context of Algorithm~\ref{SDS}. Denote $W_0=0$. For $k\geq0$, let $W_{k+1}$ be a Bernoulli random variable taking values $\log{\gamma}$ or $\log{\theta}$ with probabilities described next. For $k\geq T_\epsilon$, the probabilities are
\begin{align*}
    P(W_{k+1}=\log\gamma|\mathcal{F}_{k})&=\frac{3}{14}\\
    P(W_{k+1}=\log\theta|\mathcal{F}_{k})&=\frac{11}{14}.
\end{align*}
For $k<T_\epsilon$, when $\Delta_k>\delta_\epsilon$, the probabilities are
\begin{align*}
    P(W_{k+1}=\log\gamma|\mathcal{F}_{k},\{\Delta_k>\delta_\epsilon\})&=\frac{3}{14}\\
    P(W_{k+1}=\log\theta|\mathcal{F}_{k},\{\Delta_k>\delta_\epsilon\})&=\frac{11}{14}.
\end{align*}
For $k<T_\epsilon$, when $\Delta_k\leq\delta_\epsilon$ and $A_k=0$, we define $W_{k+1}=\log{\theta}$ with probability one
\begin{align}
    P(W_{k+1}=\log\gamma|\mathcal{F}_{k},\{\Delta_k\leq\delta_\epsilon\}\cap\{A_k=0\})&=0\label{eq25}\\
    P(W_{k+1}=\log\theta|\mathcal{F}_{k},\{\Delta_k\leq\delta_\epsilon\}\cap\{A_k=0\})&=1\notag.
\end{align}
For $k<T_\epsilon$, when $\Delta_k\leq\delta_\epsilon$ and $A_k=1$, the probabilities are
\begin{align}
    P(W_{k+1}=\log\gamma|\mathcal{F}_{k},\{\Delta_k\leq\delta_\epsilon\}\cap\{A_k=1\})&=\frac{3}{14}\frac{1}{P(A_k=1|\mathcal{F}_{k},\{\Delta_k\leq\delta_\epsilon\})}\label{eq26}\\
    P(W_{k+1}=\log\theta|\mathcal{F}_{k},\{\Delta_k\leq\delta_\epsilon\}\cap\{A_k=1\})&=1-\frac{3}{14}\frac{1}{P(A_k=1|\mathcal{F}_{k},\{\Delta_k\leq\delta_\epsilon\})}.\notag
\end{align}
Note that these last two probabilities are well defined because of Lemma~\ref{3L4}. It then follows from (\ref{eq25}), (\ref{eq26}), and the formula of total probability that for $k<T_\epsilon$
\begin{align*}
    P(W_{k+1}=\log\gamma|\mathcal{F}_{k},\{\Delta_k\leq\delta_\epsilon\})&=\frac{3}{14}\\
    P(W_{k+1}=\log\theta|\mathcal{F}_{k},\{\Delta_k\leq\delta_\epsilon\})&=\frac{11}{14}.
\end{align*}
Then we can conclude that for any $k>0$
\begin{align*}
    P(W_{k+1}=\log\gamma|\mathcal{F}_{k})&=\frac{3}{14}\\
    P(W_{k+1}=\log\theta|\mathcal{F}_{k})&=\frac{11}{14}.
\end{align*}

To be able to apply Proposition~\ref{BP} in Appendix B, the next step is to verify Assumptions~\ref{BA1}--\ref{BA2}. The validity of Assumption~\ref{BA2} results from Lemma~\ref{3L1}, where $\nu\Delta_k^2$ plays the role of $h(\Delta_k)$ in Assumption~\ref{BA2}. To complete the convergence theory, it remains to show the validity of Assumption~\ref{BA1}, which provides a lower bound for $W_{k+1}$ when $k<T_\epsilon$, by ensuring that if the stepsize $\Delta_{k+1}$ is smaller than a threshold (in our context this threshold is $\Delta_\epsilon=\delta_\epsilon\theta$), then the stepsize $\Delta_{k+1}$ must be no less than $\Delta_k\exp{(W_{k+1})}$.

\begin{lemma}\label{3L5}
One has for all $k$
\begin{align}
\mathbf{1}(k<T_\epsilon)\Delta_{k+1}\geq\mathbf{1}(k<T_\epsilon)\min\left(\Delta_ke^{W_{k+1}},\delta_\epsilon\theta\right).\label{35}
\end{align}
\end{lemma}

\begin{proof}
Note that it follows from Algorithm~\ref{SDS} and the definition of $A_k$ that $A_k=0$ if and only if $\Delta_{k+1}/\Delta_k=\theta$. When $k<T_\epsilon$ and $\Delta_k\leq\delta_\epsilon$, since we have set $W_{k+1}=\log{\theta}$ with probability one when $A_k=0$, it turns out that, with probability one, whenever the value of $\log(\Delta_{k+1}/\Delta_k)$ is $\log{\theta}$, one has $W_{k+1}=\log{\theta}$. When $A_k=1$, $\Delta_{k+1}/\Delta_k=\gamma$, and $k<T_\epsilon$, $W_{k+1}$ can take the values $\log\gamma$ and $\log\theta$. Then we conclude that when $k<T_\epsilon$ and $\Delta_k\leq\delta_\epsilon$, one has
\begin{align}
    \log(\Delta_{k+1}/\Delta_k)\geq W_{k+1}.\label{eq33}
\end{align}

Note that $\Delta_{k+1}<\delta_\epsilon\theta$ implies $\Delta_k<\delta_\epsilon$, which in turn implies $\Delta_{k+1}\geq\Delta_ke^{W_{k+1}}$ when $k<T_\epsilon$ because of~(\ref{eq33}). We have covered all possible cases, and the proof is concluded.
\end{proof}

We are now ready to apply Theorem~\ref{BP} to Algorithm~\ref{SDS} with $\Phi_0=f(x_0)-f^*+\eta\delta_0^2$.

\begin{theorem}\label{T1}
    Let Assumptions~\ref{A1}--\ref{A2} hold. Let $\theta$ and $\gamma$ be chosen such that $3\log{\gamma}+11\log{\theta}>0$. Then
\begin{align*}
    E[T_\epsilon]\leq1+\frac{14\log{\gamma}}{3\log{\gamma}+11\log{\theta}}\frac{(\gamma^2-\theta^2)(f(X_0)-f^*)+c\Delta_0^2}{\frac{c}{2}(1-\theta^2)\theta^2}\frac{(7L_f+14c)^2}{4}\frac{n}{\epsilon^2}.
\end{align*}
\end{theorem}
Theorem~\ref{T1} ensures an expected worst-case complexity bound of $O(n/\epsilon^2)$ for Algorithm~\ref{SDS}.

\section{A numerical experiment}\label{s4}

In this section, we will report the numerical performance of 
Test~\ref{ST} against Test~\ref{FT} in an experiment running probabilistic-descent direct search under Gaussian noise. Recall from Section~\ref{s23} 
that Test~\ref{ST} is a sequential hypothesis test and Test~\ref{FT} is a fixed sample test. Specifically, we ran Algorithm~\ref{SDS} with Test~\ref{ST} or Test~\ref{FT} to solve its hypothesis test problem, given a total budget of $10000$ sample evaluations.

\subsection{Experiment setup}\label{s41}

We used 38 problems  
suggested in~\cite{giovannelli2024limitation} from the CUTEst dataset~\cite{gould2015cutest,gratton2025s2mpj}, which exhibit different features in terms of non-linearity, non-convexity, and partial separability. For each problem, we generated problem instances with different dimensions, which gives us a problem set of 91 problem instances in total. Please see Table~\ref{table1} for the problem names and corresponding dimensions of our problem set. To inject noise in the objective functions, we selected the noise term $F(X,\xi)-f(x)$ to follow an independent identically distributed Gaussian distribution with variance $\sigma^2\in\{0.01,1\}$. Our choice of additive Gaussian noise is consistent with the assumption of Proposition~\ref{prop}. Other choices are certainly possible, and we discuss them in Section~\ref{s5}. For performance evaluation, we used data profiles~\cite{more2009benchmarking} and performance profiles~\cite{dolan2002benchmarking}, selecting the tolerance parameter, which defines a problem being solved, as $\tau=0.1$. Each problem instance was ran 10 times in the experiment and considered as 10 problem instances in all the profiles.

\begin{table}[H]
\centering
\begin{tabularx}{\textwidth}{|Y|Y|Y|Y|}
\hline
Name & Dimension & Name & Dimension \\ \hline
ARGLINA & 10, 50, 100 & ARGTRIGLS & 10, 50, 100 \\
ARWHEAD & 100         & BDEXP & 100 \\
BOXPOWER & 10, 100 & BROWNAL & 10, 100 \\
COSINE & 10, 100 & CURLY10 & 100 \\
DIXON3DQ & 10, 100 & DQRTIC & 10, 50, 100 \\
ENGVAL1 & 2, 50, 100 & EXTROSNB & 5, 10, 100 \\
FLETBV3M & 10, 100 & FLETCBV3 & 10, 100 \\
FLETCHBV & 10, 100 & FLETCHCR & 10, 100 \\
FREUROTH & 2, 10, 50, 100 & INDEFM & 10, 50, 100 \\
MANCINO & 10, 20, 30, 50, 100 & MOREBV & 10, 50, 100 \\
NONCVXU2 & 10, 100 & NONCVXUN & 10, 100 \\
NONDIA & 10, 50, 100 & NONDQUAR & 100 \\
PENALTY2 & 10, 50, 100 & POWER & 10, 50, 100 \\
QING & 100 & QUARTC & 25, 100 \\
SENSORS & 10, 100 & SINQUAD & 5, 50, 100 \\
SCURLY10 & 10, 100 & SCURLY20 & 100 \\
SPARSINE & 10, 50, 100 & SPARSQUR & 10, 50, 100 \\
SSBRYBND & 10, 50, 100 & TRIDIA & 10, 50, 100 \\
TRIGON1 & 10, 100 & TOINTGSS & 10, 50, 100 \\ \hline
\end{tabularx}
\caption{Names and corresponding dimensions of the 91 CUTEst problem instances in the problem set.}
\label{table1}
\end{table}

We now specify the parameters of Algorithm~\ref{SDS}. 
We used the initial point provided by CUTEst 
dataset, chose the initial stepsize $\delta_0=1$, and selected $c=0.5$. However, the choice of $\theta\in(0,1)$ and $\gamma\in(1,\infty)$ can greatly affect the performance of the algorithm. Note that it is required from Theorem~\ref{T1} that $3\log{\gamma}+11\log{\theta}>0$, otherwise, the stepsize $\delta$ of Algorithm~\ref{SDS} may still converge to $0$ even when the iterates approach a non-stationary point. We also notice that the sample size per iteration increases rapidly as the stepsize $\delta$ decreases to $0$. This phenomenon occurs for both Test~\ref{ST} and Test~\ref{FT}, and is more severe in Test~\ref{FT} with a sample size of $\Theta(\delta^{-4})$. For the purposes of both algorithm performance and a fair comparison, we want to select a large $\theta$ so that the stepsize may stay away from $0$ unless necessary. Based on the above observations, we selected $\theta=0.95$ and $\gamma=1.3$ in our numerical experiment when using either Test~\ref{ST} or Test~\ref{FT}. We tried different $\theta$ and $\gamma$ when the algorithm uses either tests and observed no significant changes in the relative performance of the methods.

To choose the parameters in Test~\ref{ST} and Test~\ref{FT}, let us suppose that at an iteration $k$ of Algorithm~\ref{SDS}, $\sigma_k^2$ is the variance of $Y_k$, or at least a known upper bound thereof. The number $C_k=c\Delta_k^2(1-\theta^2)/2(\gamma^2-\theta^2)$ in Algorithm~\ref{SDS} is known at each iteration $k$. For Test~\ref{FT}, we selected the fixed sample size $m=\sigma_k^2C_k^{-2}$. For Test~\ref{ST}, we selected the test lower and upper bounds $a_l=-b_l=\sigma_k^2/(2eC_k)$.

\subsection{Experiment results}\label{s42}

The results are reported in Figures~\ref{fig:41} and~\ref{fig:42} for the two noise variances. We tried larger budgets, different values of $\tau$ in the profiles, and smaller variances, but we observed no significant changes in the relative performance of the methods. For both cases, it can be clearly seen that using the sequential test (Test~\ref{ST}) outperforms using the fixed sample test (Test~\ref{FT}). The performance gap increases as the noise variance increases.

\begin{figure}[H]
    \centering
    \includegraphics[width=0.4\textwidth]{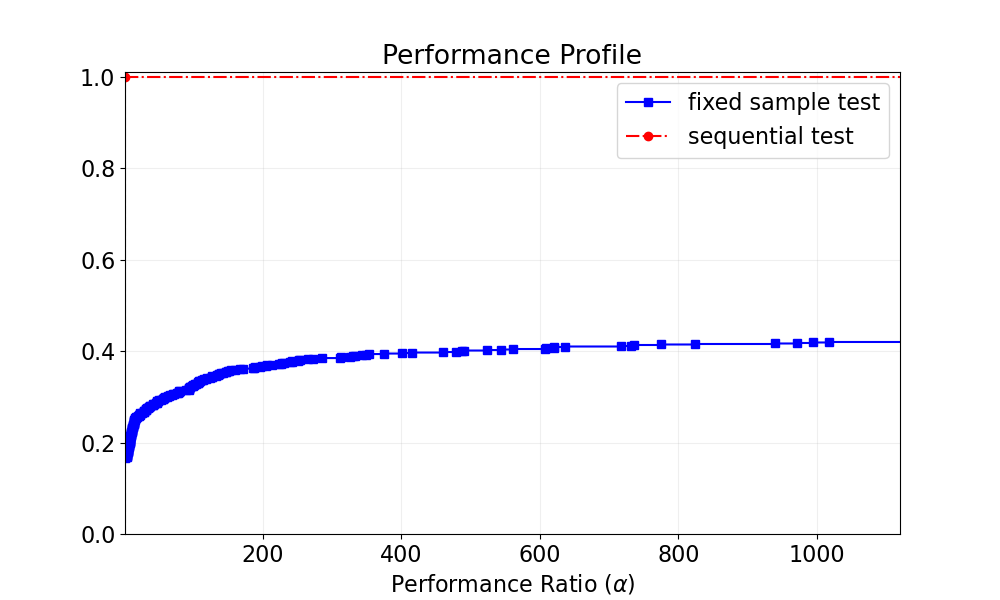}
    \includegraphics[width=0.4\textwidth]{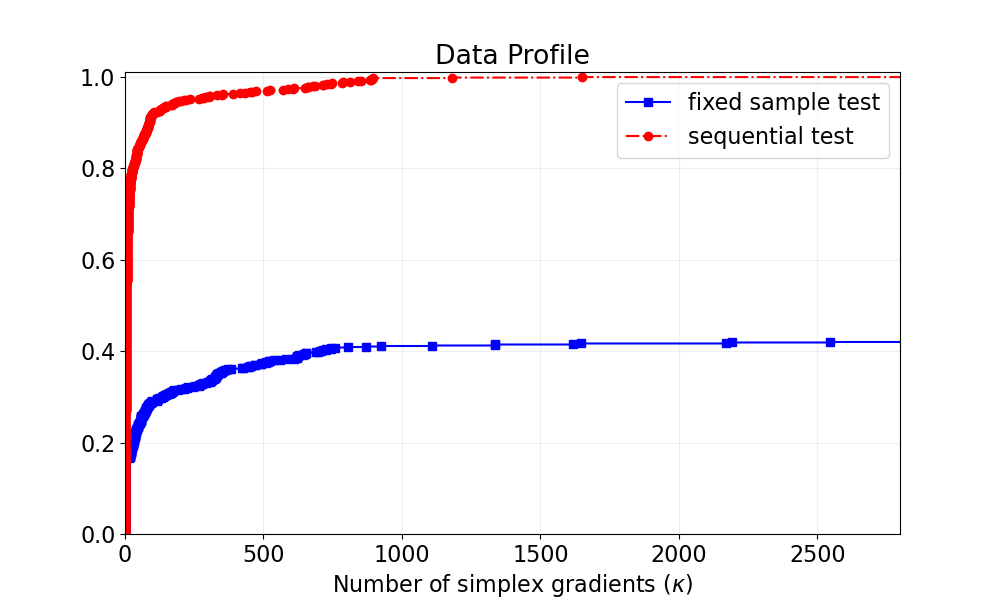}
    \caption{Performance and data profiles for Algorithm~\ref{SDS} using sequential test (Test~\ref{ST}) or fixed sample test (Test~\ref{FT}) for noise variance $\sigma^2=1$ for the problem set of Table~\ref{table1}.}
    \label{fig:41}
\end{figure}

\begin{figure}[H]
    \centering
    \includegraphics[width=0.4\textwidth]{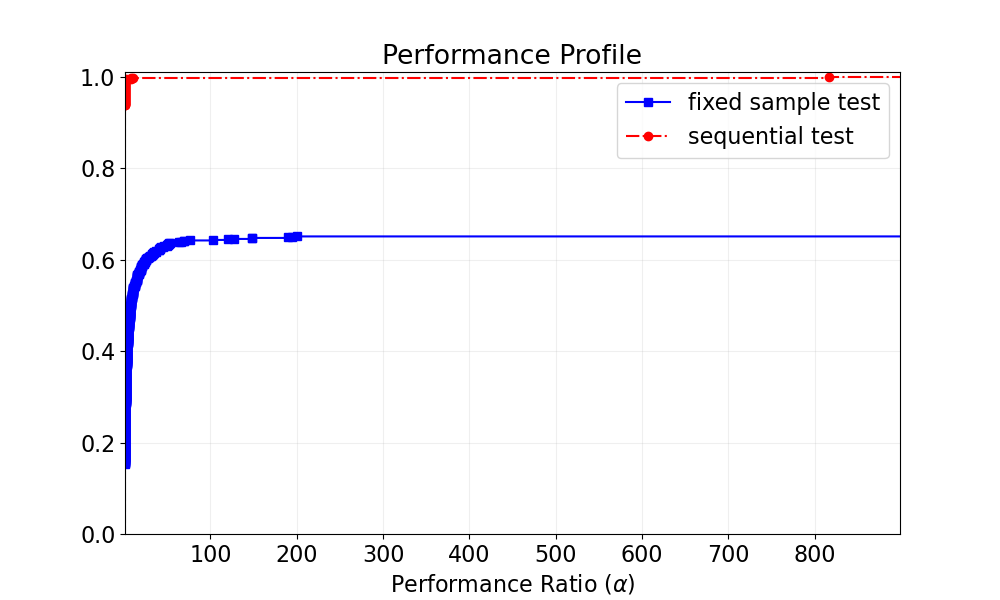}
    \includegraphics[width=0.4\textwidth]{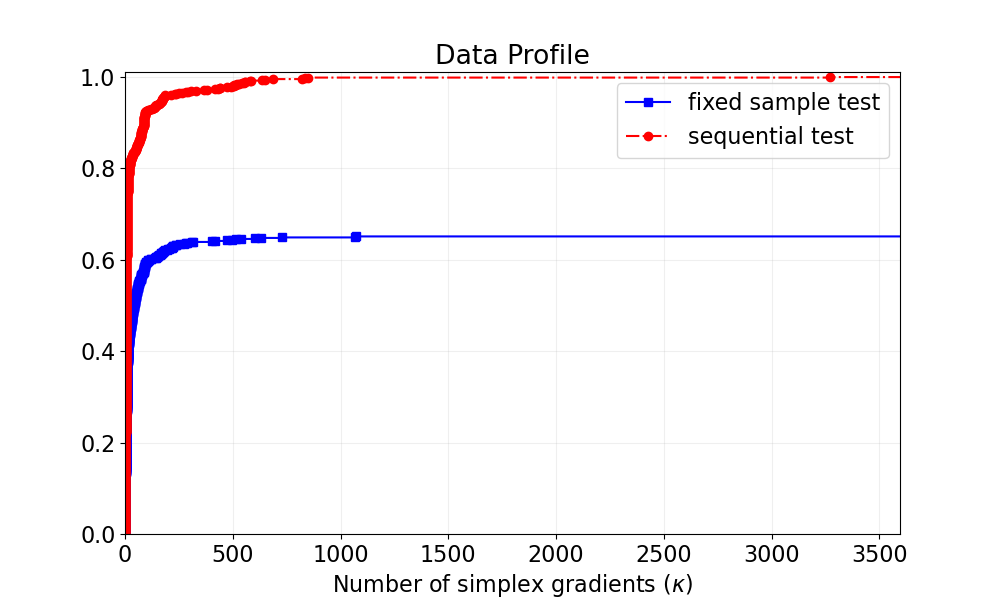}
    \caption{Performance and data profiles for Algorithm~\ref{SDS} using sequential test (Test~\ref{ST}) or fixed sample test (Test~\ref{FT}) for noise variance $\sigma^2=0.01$ for the problem set of Table~\ref{table1}.}
    \label{fig:42}
\end{figure}

\section{Concluding remarks}\label{s5}

In this paper, we introduced sequential hypothesis testing for solving a stochastic DFO problem. Specifically, we formulated the evaluation of the sufficient decrease condition~(\ref{eq2}) as a hypothesis test problem (Problem~\ref{Subp}) and solved it through a sequential hypothesis test (Test~\ref{ST}). Given an additive Gaussian noise assumption, we estimated the sample size of Test~\ref{ST} in Proposition~\ref{prop}, which indicated a possible early termination under a reduced sample size. In particular, when the potential decrease $f(x)-f(x+\delta d)$ is $\Theta(\delta^r)$ for some $r\in(0,2]$, we showed that the expected sample size can decrease from the literature's $\Theta(\delta^{-4})$ to $O(\delta^{-2-r})$.
We applied the sequential test framework to probabilistic-descent direct search and derived an iteration complexity of $O(n\epsilon^{-2})$ in Theorem~\ref{T1}. Our numerical results showed that the use of sequential hypothesis test (Test~\ref{ST}) significantly outperforms the use of its fixed sample counterpart.

The sampling procedure in the form of Test~\ref{FT} is widely used in many stochastic DFO algorithms. It is referred to in the literature of statistics
as a hypothesis test with a fixed sample size. In contrast, a sequential hypothesis test uses a random sample size, which can be regarded as a relaxation of a fixed sample size, to allow early termination of the test when the hypothesis test problem is not difficult (in our case, when the mean $\mu$ of $Y$ in~(\ref{ydefinition}) is far from 0). From this point of view, a test with a fixed sample size is some form of a restricted sequential hypothesis test. Therefore, sequential tests are generally expected to use less samples than their fixed sample size counterparts for equally good performance.

We would like to make two observations regarding the noise setting considered in this paper. The first one is that, in our numerical results in Section~\ref{s4}, 
we only show that Test~\ref{ST} performs well in a relatively limited setting where the noise is additive Gaussian. However, we can still use Test~\ref{ST} for other noise types as long as an upper bound of the noise variance of $Y$ is known. 
Regardless of the distribution of $Y$, the random walk $\sum_{i}Y^i$ is typically approximated by a Brownian motion process in the study of sequential analysis~\cite[Chapter 3.1]{siegmund2013sequential}. Therefore, we believe that Test~\ref{ST} will still work well for other noises beyond Gaussian.
The other observation is that, in this paper, we assumed knowledge of the noise variance or at least an upper bound of it. For a more practical algorithmic implementation, we may need to estimate the noise variance using some estimation techniques (see, for instance,~\cite{more2011estimating,berahas2019derivative}).

%%%%%%%%%%%%%%%%%%%%%%%%%%%%%%%%%%%%%%
\section*{Acknowledgments}
This work is partially supported by the U.S. Air Force Office of Scientific Research (AFOSR) award FA9550-23-1-0217 and the U.S. Office of Naval Research~(ONR) award~N000142412656. We thank Damiano Zeffiro and Trang H. Tran for insights and discussions about sequential sampling.

\bibliographystyle{siam}
\bibliography{reference.bib}

\begin{appendices}
\section{Appendix (Auxiliary results)}
\begin{proposition}
    Let $A>1$ be a real number. Then $\frac{A^x-1}{x(A^x+1)}\leq\frac{\log{A}}{2}$ holds for any $x\in\mathbb{R}$.
\end{proposition}

\begin{proof}
    To prove $\frac{A^x-1}{x(A^x+1)}\leq\frac{\log A}{2}$ for any $x\in\mathbb{R}$, since $\frac{A^{-x}-1}{-x(A^{-x}+1)}=\frac{A^x-1}{x(A^x+1)}$ is a symmetric function, it suffices to prove it when $x\geq0$. We rearrange and rewrite it as follows
\begin{align*}
    xA^x\log A+x\log A+2-2A^x\geq0.
\end{align*}
The derivatives of $g(x)=xA^x\log A+x\log A+2-2A^x$ are
\begin{align*}
    g'(x)&=(xA^x\log A+1-A^x)\log A\\
    g''(x)&=(xA^x\log A)(\log A)^2\geq0.
\end{align*}
It follows that $g'(x)$ is non-decreasing and $g'(x)\geq g'(0)=0$. Then $g(x)$ is non-decreasing and $g(x)\geq g(0)=0$.
\end{proof}

\begin{proposition}
    Let $t>0$ be a real number. Then $\frac{1}{t^x}\leq\frac{1}{x}$ holds for any $x\geq1$ if and only if $t\geq e^{\frac{1}{e}}$.
\end{proposition}

\begin{proof}
    Since $\frac{1}{t^x}\leq\frac{1}{x}$ holds for $x=2$, we only need to consider $t\geq\sqrt{2}$. Define $f(x)=t^x-x$. Its derivatives are $f'(x)=t^x\log{t}-1$ and $f''(x)=t^x(\log{t})^2>0$. The solution $x^*$ of $f'(x)=0$ is $x^*=-\frac{\log{\log{t}}}{\log{t}}$. If $x^*\leq1$, then $f(x)$ achieves a minimum over $[1,\infty)$ at $x=1$. We can verify that $f(1)>0$. If $x^*\geq1$, then $f(x)$ achieves a minimum over $[1,\infty)$ at $x^*$. The minimum value is $f(x^*)=\frac{1+\log{\log{t}}}{\log{t}}$. Therefore, we have that $f(x)\geq0$ holds for any $x\in[1,\infty)$ if and only if $t\geq e^{\frac{1}{e}}$.
\end{proof}

\section{Appendix (Renewal–Reward Martingale Process)}
We begin by describing how our Algorithm~\ref{SDS} can be analyzed through the lens of a renewal-reward process. A renewal event is said to occur within Algorithm~\ref{SDS} whenever the stepsize meets or exceeds a predetermined threshold, denoted by $\Delta_\epsilon$. We will show that, after one renewal, the next renewal will happen in a finite time interval, and this expected returning time is constant. Each of these renewal events is associated with a random reward, whose expectation is bounded below by a positive function $h(\Delta_\epsilon)$. Once defined in this way, the total accumulated reward across multiple renewals can be viewed as a submartingale that grows through these random increments. Since the total available reward is bounded, we will be able to deduce an expected stopping time for the algorithm mechanism. This line of analysis, which treats the underlying process as a renewal-reward martingale, was first developed in \cite{blanchet2019convergence}, where it was used to establish the expected global convergence rate of a stochastic trust-region method.

To adapt the framework~\cite{blanchet2019convergence} to our Algorithm~\ref{SDS}, we need to remove one of the assumptions from \cite[Assumption 1(ii)]{blanchet2019convergence}, which originally stipulated that $p>1/2$. We also need to show that the main result in \cite[Theorem 2]{blanchet2019convergence} still holds under more general conditions where the sequence~$W_{k+1}$ can take any positive and negative values $a$ and $b$, respectively, instead of $\pm1$ as the authors used in~\cite{blanchet2019convergence}.

Consider a stochastic process $\{(\Phi_k,\Delta_k)\}$ defined on some probability space, where $\Phi_k$ takes values in the interval $[0,\infty)$ and $\Delta_k$ takes values in the interval $(0,\infty)$, for all $k\geq0$. Let $\{W_k\}$ be another sequence of random variables defined on the same probability space as $\{(\Phi_k,\Delta_k)\}$, initialized by $W_0=0$. For all $k\geq0$, the conditional distribution of $W_{k+1}$, given the $\sigma$-algebra $\mathcal{F}_k$ generated by $\{(\Phi_0,\Delta_0,W_0),\ldots,(\Phi_k,\Delta_k,W_k)\}$, is described by
\begin{equation}\label{pdefinition}
\begin{aligned}
   P(W_{k+1}=a|\mathcal{F}_k)&=p,\\
   P(W_{k+1}=b|\mathcal{F}_k)&=1-p,
\end{aligned}
\end{equation}
where $a>0$ and $b<0$ are two constants and $p$ is the probability of taking the value $a$. When $a=1$ and $b=-1$, this construction coincides with the specific case discussed in~\cite{blanchet2019convergence}. From the above definition, it follows that $\{W_k\}$ are mutually independent and $W_k$ is also independent of the sequence $\{(\Phi_j,\Delta_j)\}_{j=0}^{k-1}$ for all $k$. Lastly, let $\{T_\epsilon\}_{\epsilon>0}$ be a family of stopping times with respect to $\{\mathcal{F}_k\}_{k\geq0}$, parameterized by some quantity $\epsilon>0$. As in~\cite{blanchet2019convergence}, we impose the following assumptions on $\{(\Phi_k,\Delta_k)\}$ and $T_\epsilon$ when $k<T_\epsilon$.

\begin{assumption}\label{BA1}
    There exists a constant $\Delta_\epsilon>0$ such that the following holds for all $k\geq0$
\begin{align*}
    \mathbf{1}(k<T_\epsilon)\Delta_{k+1}\geq\mathbf{1}(k<T_\epsilon)\min(\Delta_ke^{W_{k+1}},\Delta_\epsilon),
\end{align*}
where $W_{k+1}$ satisfies $E[W_{k+1}|\mathcal{F}_k]>0$ (which means $p(a-b)+b>0$).
\end{assumption}

\begin{assumption}\label{BA2}
    There exists a nondecreasing function $h(\cdot):[0,\infty)\rightarrow(0,\infty)$ such that
\begin{align*}
    E(\Phi_k-\Phi_{k+1}|\mathcal{F}_k)\mathbf{1}(k<T_\epsilon)\geq h(\Delta_k)\mathbf{1}(k<T_\epsilon).
\end{align*}
\end{assumption}

Assumption~\ref{BA1} tells us that for $k<T_\epsilon$, the stepsize $\Delta_k$ tends to increase and return to the threshold $\Delta_\epsilon$ when it is smaller than $\Delta_\epsilon$. Assumption~\ref{BA2} tells us that for $k<T_\epsilon$, the stochastic process $\Phi_0-\Phi_{k+1}$ acts like a submartingale and the expected martingale difference $E(\Phi_k-\Phi_{k+1}|\mathcal{F}_k)$ is at least $h(\Delta_k)$.

In order to define a renewal process, we first define an auxiliary process $\{Z_k\}_{k=0}^{\infty}$ by letting $Z_0=\log\frac{\Delta_\epsilon}{\Delta_0}$ and setting
\begin{align*}
    Z_{k+1}=\min(Z_k+W_{k+1},\log\frac{\Delta_\epsilon}{\Delta_0}),
\end{align*}
or, equivalently,
\begin{align*}
    \Delta_0e^{Z_{k+1}}=\min(\Delta_0e^{Z_k+W_{k+1}},\Delta_\epsilon).
\end{align*}
We then define the renewal process
$\{A_n\}_{n=0}^{\infty}$ by letting $A_0=0$ and setting $A_n=\inf\{m>A_{n-1}:Z_m=\log\frac{\Delta_\epsilon}{\Delta_0}\}$. From Assumption~\ref{BA1}, we have that
\begin{align*}
    \mathbf{1}(k<T_\epsilon)\Delta_{k+1}\geq\mathbf{1}(k<T_\epsilon)\min(\Delta_ke^{W_{k+1}},\Delta_\epsilon)\geq\mathbf{1}(k<T_\epsilon)\Delta_{0}e^{Z_{k+1}},
\end{align*}
where we have used a simple inductive argument to obtain the second inequality. The interarrival times of this renewal process are defined for all $n\geq1$ by
\begin{align*}
    \tau_n=A_n-A_{n-1}.
\end{align*}

The first main step in the analysis will be to bound the expected value of the interarrival time~$\tau_n$ (see Lemma~\ref{LB2}). For this purpose, one needs to bound $E[\bar{\tau}]$, where $\bar{\tau}=\inf\{n\geq0:\bar{Z}_n\geq0\}$, using the structure of the process $W_k$ (see Lemma~\ref{LB1} below).

\begin{lemma}
    Let Assumption~\ref{BA1} hold. Define the process $\bar{Z}_0=b<0$,\label{LB1} $\bar{Z}_{k+1}=\bar{Z}_{k}+W_{k+1}$ for all $k\geq0$. Then
\begin{align}
    E[\bar{\tau}]\leq\frac{a-b}{pa-pb+b}.\label{Beq1}
\end{align}
\end{lemma}

\begin{proof}
For ease of notation, let $k\wedge\bar{\tau}=\min\{k,\bar{\tau}\}$ and $v=p(a-b)+b>0$. Note that
\begin{align}
    E[W_{k+1}|\mathcal{F}_k]=v.\label{B10}
\end{align}
Consider the stochastic process defined by $R_0=\bar{Z}_0$ and for $k\geq1$
\begin{align*}
    R_k=\bar{Z}_{k\wedge\bar{\tau}}-\sum_{j=0}^{k\wedge\bar{\tau}-1}v.
\end{align*}

We first prove that $E[R_{k+1}|\mathcal{F}_k]=R_k$, which means $R_k$ is a martingale with respect to $\{\mathcal{F}_k\}$. To see this, we first note that
\begin{align*}
    R_{k+1}-R_{k}=\bar{Z}_{(k+1)\wedge\bar{\tau}}-\bar{Z}_{k\wedge\bar{\tau}}-\sum_{j=0}^{(k+1)\wedge\bar{\tau}-1}v+\sum_{j=0}^{k\wedge\bar{\tau}-1}v
\end{align*}
and
\begin{align}
    E[R_{k+1}-R_{k}|\mathcal{F}_k]=E[(R_{k+1}-R_k)\mathbf{1}(\bar{\tau}>k)|\mathcal{F}_k]+E[(R_{k+1}-R_k)\mathbf{1}(\bar{\tau}\leq k)|\mathcal{F}_k].\label{B11}
\end{align}
We now show that $E[R_{k+1}-R_{k}|\mathcal{F}_k]=0$. First, since $\left((k\wedge\bar{\tau})-((k+1)\wedge\bar{\tau})\right)\mathbf{1}(\bar{\tau}\leq k)=0$, we have
\begin{align}
    E[(R_{k+1}-R_k)\mathbf{1}(\bar{\tau}\leq k)|\mathcal{F}_k]&=E\left[\left(\bar{Z}_{(k+1)\wedge\bar{\tau}}-\bar{Z}_{k\wedge\bar{\tau}}-\sum_{j=0}^{(k+1)\wedge\bar{\tau}-1}v+\sum_{j=0}^{k\wedge\bar{\tau}-1}v\right)\mathbf{1}(\bar{\tau}\leq k)|\mathcal{F}_k\right]\notag\\
    &=E\left[0\cdot\mathbf{1}(\bar{\tau}\leq k)|\mathcal{F}_k\right]\notag\\
    &=0.\label{B12}
\end{align}
Secondly, from $\bar{Z}_{k+1}=\bar{Z}_{k}+W_{k+1}$ and~(\ref{B10}), we have
\begin{align}
    E[(R_{k+1}-R_k)\mathbf{1}(\bar{\tau}>k)|\mathcal{F}_k]=&E\left[\left(\bar{Z}_{(k+1)\wedge\bar{\tau}}-\bar{Z}_{k\wedge\bar{\tau}}-\sum_{j=0}^{(k+1)\wedge\bar{\tau}-1}v+\sum_{j=0}^{k\wedge\bar{\tau}-1}v\right)\mathbf{1}(\bar{\tau}>k)|\mathcal{F}_k\right]\notag\\
    =&E\left[\left(\bar{Z}_{k+1}-\bar{Z}_{k}-v\right)\mathbf{1}(\bar{\tau}>k)|\mathcal{F}_k\right]\notag\\
    =&E\left[\left(W_{k+1}-v\right)\mathbf{1}(\bar{\tau}>k)|\mathcal{F}_k\right]\notag\\
    =&0.\label{B13}
\end{align}
After summing up~(\ref{B12}) and~(\ref{B13}), since $R_k$ is $\mathcal{F}_k$-measurable, we have from~(\ref{B11}) that
\begin{align*}
    E[R_{k+1}|\mathcal{F}_k]=R_k.
\end{align*}

Since $R_k$ is a martingale with respect to $\{\mathcal{F}_k\}$, we immediately have $E[R_k]=R_0$. Note that $\bar{Z}_{k\wedge\bar{\tau}}\leq a$ for each $k\geq0$ due to the definition of $\bar{\tau}$ and $W_k$. We then obtain from the definition of $R_k$ that
\begin{align}
    E\left(\sum_{j=0}^{(k\wedge \bar{\tau})-1}v\right)=E[\bar{Z}_{k\wedge\bar{\tau}}]-E[R_k]\leq a-R_0=a-b.\label{Ap1}
\end{align}
Now, due to $v>0$ and $k$ being eventually larger than $\bar{\tau}$, observe that
\begin{align*}
    0<\sum_{j=0}^{(k\wedge \bar{\tau})-1}v\nearrow\sum_{j=0}^{\bar{\tau}-1}v
\end{align*}
as $k\rightarrow\infty$. Note that this conclusion holds even on the event $\{\bar{\tau}=\infty\}$. Therefore, by the monotone convergence theorem and~(\ref{Ap1}),
\begin{align*}
    E\left(\sum_{j=0}^{\bar{\tau}-1}v\right)=\lim_{k\rightarrow\infty}E\left(\sum_{j=0}^{(k\wedge \bar{\tau})-1}v\right)\leq a-b.
\end{align*}
Finally, using Wald's identity, we have
\begin{align*}
    E[\bar{\tau}]v=E\left(\sum_{j=0}^{\bar{\tau}-1}v\right)\leq a-b,
\end{align*}
which implies~(\ref{Beq1}) and concludes the proof.
\end{proof}

We will use the upper bound for the constant $E[\bar{\tau}]$ in the above lemma to give an upper bound for the constant $E[\tau_n]$.

\begin{lemma}
    Let Assumption~\ref{BA1} hold. Let $\tau_n$ be defined as before. Then for all $n$,\label{LB2}
\begin{align}
    E[\tau_n]=p+(1+E[\bar{\tau}])(1-p).\label{Beq2}
\end{align}
\end{lemma}

\begin{proof}
Note that $E[\tau_n]=E[E[\tau_n|Z_{A_{n-1}}]]=E[E[\tau_1|Z_{0}]]=E[\tau_1]$ and it suffices to verify this proposition for $n=1$.

By conditioning on $W_1$, we have that
\begin{align*}
    E[\tau_1]=1\cdot P(W_1=a)+(1+E[\bar{\tau}])P(W_1=b).
\end{align*}
This identity follows because the distribution of $\tau_1$ conditioned on $Z_1=\log\frac{\Delta_\epsilon}{\Delta_0}+b$ is the same as the distribution of $\bar{\tau}$. Thus, we simplify this expression to conclude that~(\ref{Beq2}) holds.
\end{proof}

The following proposition is proved in \cite[Theorem 2]{blanchet2019convergence} with $E[\tau_n]=p/(2p-1)$. The argument in \cite[Theorem 2]{blanchet2019convergence} works for any constant $E[\tau_n]$ and there is no need to repeat the proof here.

\begin{theorem}\label{BP}
    Let Assumptions~\ref{BA1}--\ref{BA2} hold. Then
\begin{align*}
    E[T_\epsilon-1]\leq E[\tau_n]\cdot\frac{\Phi_0}{h(\Delta_\epsilon)},
\end{align*}
\end{theorem}
where $\Phi_0$ is a given positive number, $E[\tau_n]$ satisfies~(\ref{Beq1}), and $h$ is a given function in Assumption~\ref{BA2}.
\end{appendices}

\end{document}